\documentclass{amsart}
\usepackage[cmtip,all]{xy}
\usepackage{amssymb}
\newtheorem{theorem}{Theorem}[section]

\newtheorem{cor}[theorem]{Corollary}
\theoremstyle{definition}

\theoremstyle{remark}
\newtheorem{remark}[theorem]{Remark}

\numberwithin{equation}{section}

\begin{document}

\title{On isolated log canonical centers}

\author{Chih-Chi Chou}
\address{University of Illinois at Chicago, Department of Mathematics, Statistics, and Computer Science, Chicago IL 60607  }
\address{cchou20@uic.edu}

\begin{abstract}
In this paper, we show that the depth of an isolated log canonical center
is determined by the cohomology of the -1 discrepancy diviors over it.
A similar result also holds for normal isolated Du Bois singularities.
\end{abstract}

\maketitle

\section{Introduction}
\let\thefootnote\relax\footnote{{\it 2010 Mathematics Subject Classification}: 14B05, 14F17}
Singularities play a significant role in the minimal model 
program (mmp).
Among the different types of singularities, Kawamata log terminal (klt)
and log canonical (lc) are of particular importance.
Many fundamental theorems are first proved in the klt case, then extended
to the lc case.
And it is expected that lc  should be the largest class of singularities for which one can run mmp .

One of the major differences between klt and lc is that klt singularities are rational singularities,
 lc singularities are  Du Bois \cite{KK} but in general not rational.
So it is interesting and important to  know how far lc is from being rational.
Since rational implies Cohen-Macaulay, we can also ask if the variety $X$ is 
Cohen-Macaulay at some given point $p$.
Or more precisely, we can calculate $depth_p(\mathcal{O}_{X})$.
 
There are some known results regarding this direction.
For example, Fujino shows that given a lc pair $(X, \Delta)$ of dimension
at least three, then $depth_p(\mathcal{O}_{X})\ge min\{3,\mbox{codim}_pX\}$
if $\bar{p}$ is not a lc center ( Theorem 4.21 in \cite{OF}), which is first proved by Alexeev assuming that $p$ is a closed point and $X$ is projective (Lemma 3.2 in \cite{Al}).
Koll\'{a}r and Kov\'{a}cs generalized this result in \cite{KOL} and \cite{SK}, respectively,
but still under the assumption that $\bar{p}$ is not a lc center.
(See also \cite{AH} for result about closed points.)

In this paper, we investigate  a case when $\bar{p}$ is a lc center.
Assume that $p$ is an isolated lc center, after localization we assume 
$p$ is a closed point. 
It turns out that there is a delicate relation between  $depth_p(\mathcal{O}_{X})$
and the cohomology group of the exceptional divisors over $p$.
More precisely, given a log canonical pair $(X, \Delta)$ and an isolated lc center $p\in X$ which is a closed point,  
we take a log resolution $f:Y\rightarrow X$ such that
\begin{equation*}
K_Y=f^*(K_X+\Delta)+A-B-E.
\end{equation*} 
Here $A,B$ are effective and $\lfloor B \rfloor=0$,  $E$ is the reduced  divisor such that $f(E)=p$ .
Then we have the following,
\begin{theorem}(=Corollary \ref{CM})
For any integer $3\le t\le n$, we have depth$_p\mathcal{O}_X\ge t$  if and only if 
$\mbox{H}^{i-1}(E,\mathcal{O}_{E})=0, \forall 1<i<t.$ (Note that by assumption $X$ is normal, so we know depth$_p\mathcal{O}_X$ is at least two.)
\end{theorem}

This result generalizes Proposition 4.7 in \cite{OFI}, which gives a necessary and sufficient condition for an index one isolated log canonical singularity to be Cohen-Macaulay.   

We prove this theorem by showing that the local cohomology $\mbox{H}^i_p(\mathcal{O}_{X})$ is the Matlis dual of $\mbox{H}^{n-i}(E, K_E)$.
The same method applies to isolated Du Bois singularities, (see section 3.2). 
In the Du Bois case, $E$ denotes the reduced exceptional divisors.

The most crucial ingredient  of the proof is Kov\'{a}cs vanishing theorem,
which says that $R^if_*\mathcal{O}_Y(-E)=0, \forall i>0.$
With this theorem, we see that $f_*\mathcal{O}_Y(-E)$ is quasi isomorphic 
to $\mathcal{R}f_*\mathcal{O}_Y(-E)$.
By this quasi isomorphism and Grothendieck duality,  we are able  to see the relation between the local cohomology of $X$ and cohomology of $\mathcal{O}_E$.
Because of the significant role of Kov\'{a}cs's theorem in this paper, 
we give a quick proof of it  in the last section.  
This proof, based on Fujino's idea, only  uses Grothendieck duality and Kawamata-Viehweg vanishing theorem 
instead of  the notion of Du Bois pair in Kov\'{a}cs's original paper.

{\it Acknowledgements.}
I would like to thank  professor Osamu Fujino for discussions 
and answering many questions   by emails.
Moreover, this project is inspired by his paper \cite{OFI}.  
I would also like to thank professors Lawrence Ein, S\'{a}ndor  Kov\'{a}cs and Mihnea Popa for many useful discussions. 
I am also grateful to the referee for careful reading and many
useful comments.

\section{preliminaries}

Given a  pair $(X, \Delta)$, where $X$ is a normal variety and $\Delta$ is 
a $\mathbb{Q}-$linear combination of Weil divisors  so that $K_X+\Delta$ is $\mathbb{Q}$-Cartier.
Take a log resolution $f:Y\rightarrow X$, such that the exceptional locus 
and the strict transform $f_*^{-1}\Delta$ are simple normal crossing divisors.
We say the pair $(X, \Delta)$ is log canonical if 
\begin{equation*}
K_Y=f^*(K_X+\Delta)+A-B-E,
\end{equation*} 
where $A, B$ are effective, $\lfloor B \rfloor=0$ and $E$ is reduced.
We say $(X, \Delta)$ is  log terminal if $E$ is empty.

In this paper we consider log canonical pair, $(X,\Delta)$.
A sub variety $W\subset X$ is called log canonical center,
if there is a log resolution as above, and some component $E'\subset E$  
such that $f(E')=W$.

We recall Kov\'{a}cs vanishing theorem.
\begin{theorem} (Theorem 1.2 in \cite{KOV})\label{KV}
Let $(X, \Delta)$ be a log canonical pair and let $f:Y\rightarrow X$ be a proper birational morphism from a smooth variety $Y$ such that $\mbox{Ex}(f) \cup$ $\mbox{Supp} f_{*}^{-1}\Delta$ is a simple normal crossing divisor on $Y$.
If we write 
\begin{equation*}
K_Y=f^*(K_X+\Delta)+\sum _i a_iE_i
\end{equation*}
and put $E=\sum _{a_i=-1}E_i$, then 
\begin{equation*}
R^if_*\mathcal{O}_Y(-E)=0 
\end{equation*}
for every $i>0$.
\end{theorem}
This theorem is first proved by notion of Du Bois pair under the assumption that 
$X$ is $\mathbb{Q}$-factorial.
The proof  is then simplified  in \cite{OF2} without assuming  $\mathbb{Q}$-factorial.

Now  we recall the duality theorems which will be used in this paper.
First we recall Grothendieck duality theorem (III.11.1, VII.3.4 in \cite{Har66}).
Let $f:Y\rightarrow X$ be a proper morphism between finite
dimensional noetherian schemes.
Suppose that both $X$ and $Y$ admit dualizing complexes, for example when they are quasi-projective varieties.
Then for any $\mathcal{F}^{\bullet}\in D^{-}_{qcoh}(Y)$, we have 
\begin{equation*}
Rf_*R\mathcal{H}om_Y(\mathcal{F}^{\bullet}, \omega ^{\bullet}_Y)\cong R\mathcal{H}om_X(Rf_*\mathcal{F}^{\bullet},\omega ^{\bullet}_X)
\end{equation*}
Here $\omega ^{\bullet}_X$ is  dualizing complex. 
Let $n$ be the dimension of $X$ and assume that $X$ is normal, then $h^{-n}(\omega ^{\bullet}_X):= \omega_X =\mathcal{O}_X(K_X) $, the extension of  regular $n$-forms on smooth locus.
In this paper we only consider normal varieties, so we will use $\omega_X$ and
$K_X$ interchangeably.
If $X$ is Cohen-Macaulay, then $h^i(\omega ^{\bullet}_X)=0, $ if $i\neq -n$, and 
$h^{-n}(\omega ^{\bullet}_X)=\omega _X$.
Or equivalently, $\omega ^{\bullet}_X=\omega_X[n].$

Now we recall local duality (V.6.2 in \cite{Har66}). 
Suppose that $(R,p)$ is a local
ring. An injective hull $I$ of the residue field $k=R/p$ is a an injective $R$ module $I$ 
such that 
for any non-zero submodule $N\subset I$ we have $N \cap k\neq 0$.
(See \cite{BH} Proposition 3.2.2. for more discussion.)
Matlis duality says that the functor $Hom(\cdot, I)$ is a faithful exact 
functor on the category of Noetherian $R$ modules.

\begin{theorem}(Local duality )
Let $(R,p)$ be a local ring and $\mathcal{F}^{\bullet}\in D^{+}_{coh}(R)$.
Then 
\begin{equation*}
\mathbf{R}\Gamma_p(\mathcal{F}^{\bullet})\rightarrow \mathbf{R}Hom
(\mathbf{R}Hom(\mathcal{F}^{\bullet}, \omega _R ^{\bullet}), I)
\end{equation*}
is an isomorphism.
\end{theorem}

In particular, if we take $i$-th cohomology on both hand sides,
we have
\begin{equation*}
\mbox{H}^i_p(\mathcal{F}^{\bullet})\cong Hom (\mbox{H}^{-i}(\mathbf{R}Hom(\mathcal{F}^{\bullet}, \omega _R ^{\bullet})),I)
\end{equation*}
The $-i$ comes from switching the cohomology functor $\mbox{H}^i(\cdot)$ and $Hom(\cdot, I)$.

\section{ Main Results}
\subsection{Depth of LC center}
Given a log canonical pair $(X, \Delta)$, and an isolated lc center $p\in X$ which is a closed point.
Without loss of generality, we assume $X$ is an affine space and $p$ is the only closed point.
By definition, we have a log resolution $f:Y\rightarrow X$ such that
\begin{equation*}
K_Y=f^*(K_X+\Delta)+A-B-E.
\end{equation*} 
Here $A,B$ are effective and $\lfloor B \rfloor=0$,  $E$ is the reduced exceptional divisor such that $f(E)=p$ .

\begin{theorem}\label{dual}
For $1<i<n$, $\mbox{H}_p^{i}(X,\mathcal{O}_X)$ is dual to $\mbox{H}^{n-i}(E, K_E)$ by Matlis duality.
For $i=n$, $\mbox{H}_p^n(X,\mathcal{O}_X)$ is dual to $f_*\mathcal{O}_Y(K_Y+E)$.

\end{theorem}
\begin{proof}
Push forward  the following exact sequence on $Y$,
\begin{equation*}
0\rightarrow K_Y\rightarrow K_Y(E)\rightarrow K_{E}\rightarrow 0.
\end{equation*}
By Grauert-Riemenschneider vanishing, we have $R^{n-i}f_*\mathcal{O}_Y(K_Y+E)\cong \mbox{H}^{n-i}(E, K_E)$ for $i<n$ .
So to prove the statement, it suffices to prove the duality between  $\mbox{H}_p^i(X,\mathcal{O}_X)$ and
  $R^{n-i}f_*\mathcal{O}_Y(K_Y+E)\cong \mbox{H}^{n-i}(E, K_E)$.
To this end, we consider the quasi isomorphism $f_*\mathcal{O}_Y(-E)\cong _{quasi}\mathbf{R}f_*\mathcal{O}_Y(-E)$  implied by  Kov\'{a}cs vanishing theorem.
Apply Grothendieck duailty, we have
\begin{equation*}
\mathbf{R}Hom(f_*\mathcal{O}_Y(-E), \omega _X ^{\bullet})\cong_{quasi}
\mathbf{R}Hom(\mathbf{R}f_*\mathcal{O}_Y(-E), \omega _X ^{\bullet})\cong_{quasi}\mathbf{R}f_*\omega_Y^{\bullet}(E)
\end{equation*}
Take $-i$th cohomology, we have 
\begin{equation}\label{one}
Ext^{-i}(f_*\mathcal{O}_Y(-E), \omega _X ^{\bullet})\cong R^{n-i}f_*\mathcal{O}_Y(K_Y+E)
\end{equation}
By Matlis duality, the left hand side is isomorphic to Hom$(\mbox{H}_p^i(f_*\mathcal{O}_Y(-E)), I )$, where $I$ is injective hull of $k$ .

To prove the  statement, we claim that  $\mbox{H}_p^i(f_*\mathcal{O}_Y(-E))\cong \mbox{H}_p^i(\mathcal{O}_X)$ for $i>1$.
This follows from the following exact sequence 
\begin{equation*}
0\rightarrow f_*\mathcal{O}_Y(-E)\rightarrow \mathcal{O}_X \rightarrow \mathcal{O}_p\rightarrow 0,
\end{equation*}
and the fact that $\mbox{H}^i(\mathcal{O}_p)= 0$ iff $i>0.$

\end{proof}

\begin{cor}\label{CM}
For any integer $3\le t\le n$, we have depth$_p\mathcal{O}_X\ge t$  if and only if 
$\mbox{H}^{i-1}(E,\mathcal{O}_{E})=0, \forall 1<i<t.$ (Note that by assumption $X$ is normal, so we know depth$_p\mathcal{O}_X$ is at least two.)
\end{cor}

\begin{proof}
In the proof of Theorem \ref{dual}, we showed $\mbox{H}_p^i(f_*\mathcal{O}_Y(-E))\cong \mbox{H}_p^i(\mathcal{O}_X)$ for $i>1$.  
So for any integer $3\le t\le n$, we have  
\begin{eqnarray*}
\mbox{depth}_p\mathcal{O}_ X\ge t &\Leftrightarrow& \mbox{H}_p^i(X,\mathcal{O}_X)=0,\forall i<t\\
&\Leftrightarrow & \mbox{H}^i_p(X, f_*\mathcal{O}_Y(-E))=0, \forall 1<i<t\\
&\Leftrightarrow& \mbox{H}^{n-i}(E,K_{E})=0, \forall 1<i<t (\mbox{Matlis duality and Equation (\ref{one}))}\\
&\Leftrightarrow& \mbox{H}^{i-1}(E,\mathcal{O}_{E})=0,  \forall 1<i<t.(\mbox{Serre Duality})
\end{eqnarray*}
\end{proof}

\begin{remark}
The cohomology group $\mbox{H}^i(E, \mathcal{O}_{E})$ is independent of resolution, because  
$\mbox{H}^i(E, \mathcal{O}_{E})\cong R^if_*\mathcal{O}_Y$ by Kov\'{a}cs vanishing theorem.
And that $R^if_*\mathcal{O}_Y$ is well known to be independent of resolution.
\end{remark}

 \begin{cor} (Proposition 4.7 \cite{OFI})
 Given a closed isolated lc center $p$ of a pair $(X,\Delta)$, then $X$ is Cohen-Macauley at $p$ if and only if 
 $\mbox{H}^i(E, \mathcal{O}_{E})=0, \forall 0<i<n-1$. 
 \end{cor}

\subsection{Normal isolated Du Bois singularity}

The notion of Du Bois singularities is a generalization of the notion of rational singularities.
 For a proper scheme of finite type $X$ there exists a complex $\underline{\Omega}_X^{\bullet}$,
 which is an analogue of De Rham complex.  
Roughly speaking , $X$ is said to have Du Bois singularities if the natural map $\mathcal{O}_X\rightarrow \underline{\Omega}_X^0 $ is a quasi isomorphism.
We refer the reader to \cite{KS11} and the reference there for more discussions.

In this subsection we consider the case where $(X, p) $ is a normal isolated Du Bois singularity of dimension $n$,
and $f:Y \rightarrow X$ is a log resolution such that $f$ is an isomorphism outside of $p$.
We claim that the idea in the previous subsection can be applied to this case.
The crucial fact we need is the following,

\begin{theorem} (Theorem 6.1 in \cite{KS11})
Take a log  resolution $f:Y\rightarrow X$ as above, and let $E$ be the reduced preimage of $p$. 
Then $(X,p)$ is a normal Du Bois singularity if and only if the natural map
\begin{equation*}
R^if_*\mathcal{O}_Y\rightarrow R^if_*\mathcal{O}_E
 \end{equation*} 
 is an isomorphism for all $i>0$.
\end{theorem}

This theorem implies that $R^if_*\mathcal{O}_Y(-E)=o, \forall i>0.$
That is 
\begin{equation*}
f_*\mathcal{O}_Y(-E)\cong_{quasi}\mathbf{R} f_*\mathcal{O}_Y(-E)
\end{equation*}
Then exactly the same proof as in previous section yields

\begin{theorem}
Given $(X, p) $ is a normal isolated Du Bois singularity of dimension $n$.
For $1<i<n$, $\mbox{H}_p^{i}(X,\mathcal{O}_X)$ is dual to $\mbox{H}^{n-i}(E, K_E)$ by Matlis duality.
For $i=n$, $\mbox{H}_p^n(X,\mathcal{O}_X)$ is dual to $f_*\mathcal{O}_Y(K_Y+E)$.
In particular, $f_*\mathcal{O}_Y(K_Y+E)\cong K_X$.
\end{theorem}

Then the corollaries in the previous section also hold.

\begin{remark}
The last statement has been proved in \cite{Ishii} (the Claim in Theorem 2.3).
\end{remark}

\section{Kov\'{a}cs vanishing theorem}

In this section we follow Fujino's idea to give a simple proof 
of Kov\'{a}cs vanishing theorem.
First we prove a similar result for dlt pair which was proved by the notion 
of rational pair in \cite{Kol13}.
One of the equivalent definitions of dlt singularities is that there is a log resolution (Szab\'{o} resolution \cite{SZA}) $f:Y\rightarrow X$ such that the dicrepancy $a(E; X,\Delta)>-1$ for any exceptional divisor $E$ on $Y$ (Theorem 2.44 in \cite{KM}).

\begin{theorem} \label{kvanish}
Let $(X, \Delta _X)$ be a dlt pair and let $f:Y\rightarrow X$ be a Szab\'{o} resolution.
Then  we can write 
\begin{equation*}
K_Y +\Delta _Y=f^*(K_X+\Delta_X)+A-B,
\end{equation*}
where $A,B$ are effective exceptional  divisors, $\lfloor B \rfloor=0$ and $\Delta _Y$ is the strict transform of $\Delta _X$.
Then for any reduced subset $\Delta ' \subseteq \Delta_Y$, 
 we have
\begin{equation*}
R^if_*\mathcal{O}_Y(-\Delta ')=0 
\end{equation*}
for every $i>0$.
\end{theorem}

\begin{proof}
Write 
\begin{equation*}
K_Y- f^*(K_X+\Delta _X)+\Delta _Y=A-B,
\end{equation*}
Then
\begin{equation*}
\lceil A\rceil=K_Y- f^*(K_X+\Delta _X)+\Delta _Y+B+\lceil A\rceil-A,
\end{equation*}
which is  f-exceptional and effective.
Consider the following diagram of complexes, 

\centerline{$\xymatrix{ 
f_*\mathcal{O}_Y(-\Delta ') \ar[r]^{\alpha} \ar[rd]^{\beta}
&\mathbf{R}f_*\mathcal{O}_Y(-\Delta ') \ar[d]^{\gamma}  \\ 
&\mathbf{R}f_*\mathcal{O}_Y(\lceil A\rceil-\Delta ')}$}

Note that 
\begin{equation*}
\lceil A\rceil-\Delta'= K_Y-f^*(K_X+\Delta _X)+\mbox{strict transform}+\delta,
\end{equation*}
where $\delta$ is some effective simple normal crossing divisors such that $\lfloor \delta \rfloor=0$.
So by Reid-Fukuda type vanishing $R^if_*\mathcal{O}_Y(\lceil A\rceil-\Delta')=0$ for $i>0$.
On the other hand, Since $\lceil A\rceil$ is exceptional and $\Delta'$ is strict transform, so $f_*\mathcal{O}_Y(\lceil A\rceil- \Delta' )=f_*\mathcal{O}_Y(-\Delta ')$. (Lemma 12 in \cite{KOL}).
So $\beta$ is a quasi isomorphism. 

Dualize this diagram we have 
 
\centerline{\xymatrix{ 
\mathbf{R}Hom(f_*\mathcal{O}_Y(-\Delta'), \omega_X^{\bullet})  
&\mathbf{R}Hom(Rf_*\mathcal{O}_Y(-\Delta'), \omega_X^{\bullet})  \ar[l]_{\alpha ^*} \\ 
&\mathbf{R}Hom(Rf_*\mathcal{O}_Y(\lceil A\rceil-\Delta'), \omega_X^{\bullet}) \ar[u]\ar[lu]^{\beta ^*}  \ar[u]_{\gamma ^*}
}}

Apply Grothendieck duality we get the following composition,

\centerline{\xymatrix{
  \mathbf{R}f_*\omega_Y^{\bullet}(\Delta'-\lceil A\rceil) \ar[r]^(.55){\gamma^*} \ar@<2pt>@/_1.5pc/[rr]_{\beta^*} &  \mathbf{R}f_*\omega_Y^{\bullet}(\Delta') \ar[r]^(.38){\alpha ^*} & \mathbf{R}\mathcal{H}om(f_*\mathcal{O}_Y(-\Delta'), \omega_X^{\bullet})
}}

By Reid-Fukuda type vanishing, the complex $ \mathcal{R}f_*\omega_Y^
{\bullet}(\Delta')$ has vanishing higher cohomology.
Note that $\beta^*$ is a quasi isomorphism, so it is in fact a composition of sheaf morphisms as following,

\centerline{\xymatrix{
  f_*\omega_Y(\Delta'-\lceil A\rceil) \ar[r]^(.55){\gamma^*} &  f_*\omega_Y(\Delta') \ar[r]^(.38){\alpha ^*} &Hom(f_*\mathcal{O}_Y(-\Delta'), \omega_X)
}}
Since $\lceil A\rceil$ is effective, $\gamma^*$ is injective.
 Because $f_*\omega_Y(\Delta')$ is a rank one sheaf and the composition $\alpha^*\circ \gamma^*$ is an isomorphism,
 $\gamma^*$ is in fact isomorphism.
This implies that $\alpha ^*$ is a quasi isomorphism, so $\alpha: f_*\mathcal{O}_Y(-\Delta ') \rightarrow \mathbf{R}f_*\mathcal{O}_Y(-\Delta ')$ is also a  quasi isomorphism.
That is, $R^if_*\mathcal{O}_Y(-\Delta ')=0, \forall i>0 $.

\end{proof}

With theorem \ref{kvanish}, we can prove Kov\'{a}cs vanishing theorem
following Fujino's idea \cite{OF2}.
Consider the following maps,

\centerline{\xymatrix{
  Y \ar[r]^(.55){h}\ar@/_.8pc/[rr]_{f} &  Z \ar[r]^(.38){g} & X
}}
where $g:(Z, \Delta_Z)\rightarrow (X, \Delta)$ is a dlt modification 
  such that $K_Z+\Delta_Z=g^*(K_X+\Delta)$.
And
$h:Y\rightarrow Z$ is a Szab\'{o} resolution such that 
$K_Y=h^*(K_Z+\Delta_Z)+A-B-\Delta_Y$, where $\Delta_Y=h^{-1}_*\Delta_Z.$

We claim that $R^if_*\mathcal{O}_Y(-\lfloor\Delta _Y\rfloor)=0, \forall i>0$.
By theorem \ref{kvanish}, $R^ih_*\mathcal{O}_Y(-\lfloor\Delta _Y\rfloor)=0, \forall i>0$.
Also note that $h_*\mathcal{O}_Y(\lceil A \rceil-\lfloor\Delta _Y\rfloor)=h_*\mathcal{O}_Y(-\lfloor\Delta _Y\rfloor)=\mathcal{O}_Z(-\lfloor \Delta_Z \rfloor)$. (lemma 12 in \cite{KOL}).
So by Leray spectral sequence, $R^if_*\mathcal{O}_Y(-\lfloor\Delta _Y\rfloor)=R^ig_*\mathcal{O}_Z(-\lfloor \Delta_Z \rfloor)$.
The latter is zero for $i>0$ by Theorem 1-2-5 and Remark 1-2-6  in \cite{KMM}.

Note that $f:Y\rightarrow X$ is not a log resolution.
To fix the problem, we can blow up centers with simple normal crossing 
with Supp$( \Delta_Y +A+B)$.
Say the blow up is $\pi: W\rightarrow Y$.
There are two cases can happen; 
If we blow up klt locus,  it is a Szab\'{o} resolution and then 
the divisor with $-1$ discrepancy is $\Delta _W=\pi^{-1}_*(\Delta _Y)$.
Then $R^i\pi_*\mathcal{O}_W(-\Delta _W)=0$ by theorem \ref{kvanish}.
If we blow up centers inside non-klt locus, then 
the divisor with $-1$ discrepancy may be $\Delta _W=\pi^{-1}_*(\Delta _Y)+F$, where $F$
is the exceptional divisor produced by blow up.
Then  $R^i\pi_*\mathcal{O}_W(-\Delta _W)=0$ by  direct calculation.
In any case we showed that the higher direct image is not changed by these two kinds of blowing up.
So we can conclude Kov\'{a}cs vanishing theorem.

\bibliographystyle{amsalpha}

\end{document}